\documentclass[11pt]{amsart}
\usepackage{geometry}                
\usepackage{graphicx}
\usepackage{amssymb}
\usepackage{epstopdf}
\usepackage[all]{xypic}
\usepackage{palatino}

\usepackage{hyperref}
\usepackage[arrow,curve,matrix]{xy}

\renewcommand{\deg}{{\rm deg}}

\newcommand{\rk}{{\rm rk}}

\newcommand{\dblq}{{/\!/}}
\theoremstyle{plain}

\newtheorem{thm}{Theorem}[section]

\newtheorem{cor}[thm]{Corollary}
\newtheorem{prop}[thm]{Proposition}
\newtheorem{lem}[thm]{Lemma}

\newtheorem{con}[thm]{Conjecture}

\theoremstyle{definition}

\newtheorem{rmk}[thm]{Remark}

\def\PP{{\textbf P}}
\def\OO{\mathcal{O}}

\def\F{\mathcal{F}}

\def\E{\mathcal{E}}

\def\cM{\mathcal{M}}

\def\cU{\mathcal{U}}

\def\mm{\overline{\mathcal{M}}}

\def\ss{\mathcal{S}}

\author[M. Aprodu]{Marian Aprodu}
\address{Institute of Mathematics "Simion Stoilow" of the Romanian
Academy, Calea Grivi\c tei \indent 21, Sector 1, RO-010702 Bucharest, Romania}
 \email{{\tt
aprodu@imar.ro}}

\author[G. Farkas]{Gavril Farkas}
\address{Humboldt-Universit\"at zu Berlin, Institut f\"ur Mathematik,  Unter den Linden 6
\hfill \newline\texttt{}
 \indent 10099 Berlin, Germany} \email{{\tt farkas@math.hu-berlin.de}}

 \author[A. Ortega]{Angela Ortega}
\address{Humboldt-Universit\"at zu Berlin, Institut f\"ur Mathematik,  Unter den Linden 6
\hfill \newline\texttt{}
 \indent 10099 Berlin, Germany} \email{{\tt ortega@math.hu-berlin.de}}

\title{Restricted Lazarsfeld-Mukai bundles and canonical curves}

\begin{document}

\dedicatory{Dedicated to Professor Shigeru Mukai on his sixtieth birthday, with admiration}
\maketitle


\vskip 5pt

For a $K3$ surface $S$, a smooth curve $C\subset S$ and a globally generated linear series $A\in W^r_d(C)$ with $h^0(C, A)=r+1$, the \emph{Lazarsfeld-Mukai} vector bundle $E_{C, A}$ is defined via the following elementary modification on $S$
\begin{equation}
\label{eqn: F}
0\longrightarrow E_{C, A}^{\vee} \longrightarrow H^0(C, A)\otimes \OO_S\longrightarrow A\longrightarrow 0.
\end{equation}

The bundles $E_{C, A}$ have been introduced more or less simultaneously in the 80's by Lazarsfeld \cite{L1} and Mukai \cite{M89} and have acquired quite some prominence in algebraic geometry. On one hand, they have been used to show that curves on general $K3$ surfaces verify the Brill-Noether theorem \cite{L1}, and this is still the only class of smooth curves known to be general in the sense of Brill-Noether theory in every genus. When $\rho(g, r,d)=0$, the vector bundle $E_{C,A}$ is rigid and plays a key role in the classification of Fano varieties of coindex $3$. For $g=7, 8, 9$, the corresponding Lazarsfeld-Mukai bundle has been used to coordinatize the moduli space of curves of genus $g$ , thus giving rise to a new and more concrete model of $\cM_g$, see  \cite{M2}, \cite{M3}, \cite{M4}. Furthermore,  Lazarsfeld-Mukai bundles of rank two have led  to a characterization of the locus in $\cM_g$ of curves lying on $K3$ surfaces in terms of existence of linear series with unexpected syzygies \cite{F}, \cite{V}. For a recent survey on this circle of ideas, see \cite{A}.

\vskip 3pt

Recently, Lazarsfeld-Mukai bundles have proven to be effective in shedding some light on an interesting conjecture of Mercat  in Brill-Noether theory, see  \cite{FO1}, \cite{FO12}, \cite{LMN}. Recall that the Clifford index of a semistable vector bundle $E\in \cU_C(n, d)$ on a smooth curve $C$ of genus $g$ is defined in \cite{LN1} as
$$\gamma(E):=\mu(E)-\frac{2}{n} h^0(C, E)+2\geq 0.$$
Then the  \emph{higher Clifford indices} of the curve $C$ are defined as the quantities
$$\mathrm{Cliff}_n(C):=\mbox{min}\Bigl\{\gamma(E): E\in \cU_C(n, d), \ \ d\leq n(g-1), \ \ h^0(C, E)\geq 2n\Bigr\}.$$
For any line bundle $L$ on $C$ such that $h^i(C, L)\geq 2$ for $i=0, 1$, that is, contributing to the classical Clifford index $\mbox{Cliff}(C)$, by computing the invariants of the strictly semistable vector bundle $E:=L^{\otimes n}$, one finds that
$\mbox{Cliff}_n(C)\leq \mbox{Cliff}(C).$
Mercat \cite{Me1} predicted that for any smooth curve $C$ of genus $g$, the following equality
$$
(M_n):  \  \   \mathrm{Cliff}_n(C)=\mathrm{Cliff}(C).
$$
should hold. Counterexamples to $(M_2)$ have been found on curves lying on $K3$ surfaces that are special in Noether-Lefschetz sense, see \cite{FO1}, \cite{FO12} and \cite{LN2}. However, $(M_2)$ is expected to hold for a general curve of genus $g$, and in fact even for a curve $C$ lying on a $K3$ surface $S$ such that $\mbox{Pic}(S)=\mathbb Z \cdot C$. For instance, it is known that statement $(M_2)$ holds on $\cM_{11}$ outside a certain Koszul divisor (which also admits a Noether-Lefschetz realization), see \cite{FO12} Theorem 1.3. It has also been shown that $(M_2)$ holds generically on $\cM_g$ for $g\leq 16$, see \cite{FO1}.

\vskip 3pt

It has been proved in \cite{LMN} that rank three restricted Lazarsfeld-Mukai bundles invalidate statement $(M_3)$ in genus $9$ and $11$ respectively, that is, Mercat's conjecture in rank three fails generically on  $\cM_9$ and $\cM_{11}$ respectively. This was then extended in \cite{FO12} Theorem 1.4, to show that on a $K3$ surface $S$ with $\mbox{Pic}(S)=\mathbb Z\cdot C$, where $C^2=2g-2$, if $A\in W^2_d(C)$ is a linear system where $d:=\lfloor \frac{2g+8}{3}\rfloor$, the restriction to $C$ of the Lazarsfeld-Mukai bundle $E_{C,A}$ is stable and has Clifford index strictly less than $\lfloor \frac{g-1}{2}\rfloor $, in particular, statement $(M_3)$ fails for the curve $C$.  For further background on this problem, we refer to the papers \cite{Me1},  \cite{LN1} and \cite{GMN}.

\vskip 3pt

The restricted Lazarsfeld-Mukai bundle $E|_C:=E_{C, A}\otimes \OO_C$ sits in the following exact sequence on the curve $C\subset S$
\begin{equation}
\label{eqn: E|_C}
0\longrightarrow Q_A\longrightarrow E|_C\longrightarrow K_C\otimes A^{\vee}\longrightarrow 0,
\end{equation}
where $Q_A=M_A^{\vee}$ is the dual of the kernel bundle defined by the sequence
$$0\longrightarrow M_A\longrightarrow H^0(C,A)\otimes \OO_C\longrightarrow A\longrightarrow 0.$$ One then easily shows \cite{V}, \cite{FO12} that the sequence (\ref{eqn: E|_C}) is exact on global sections, that is,
$$h^0(C, E|_C)=h^0(C, K_C\otimes A^{\vee})+h^0(C, Q_A)=g-d+2r+1.$$ By choosing the degree $d$ minimal such that $W^r_d(C)\neq \emptyset$, precisely $d=r+\bigl\lfloor \frac{r(g+1)}{r+1}\bigr\rfloor$, it becomes clear that, for sufficiently high $g$, one has $$\gamma(E|_C)<\mathrm{Cliff}(C),$$ that is,  $E|_C$, when semistable, provides a counterexample to Mercat's conjecture $(M_{r+1})$. We prove the following result, extending to rank $4$ a picture studied in smaller ranks in the papers \cite{M89}, \cite{V}, respectively \cite{FO12}.

\begin{thm}
\label{thm: rank four}
Let $S$ be a $K3$ surface  with $\mathrm{Pic}(S)=\mathbb Z\cdot L$, where $L^2=2g-2$ and
write $$g=4i-4+\rho \ \mbox{ and }   \  \ d=3i+\rho,$$ with $\rho\geq 0$ and $i\geq 6$. Then for a general curve
$C\in |L|$ and a globally generated linear series $A\in W^3_d(C)$ with $h^0(C, A)=4$, the restriction to $C$ of the Lazarsfeld-Mukai bundle $E_{C, A}$ is stable.
\end{thm}

Note that in Theorem \ref{thm: rank four}, $\mbox{dim } W^3_d(C)=\rho$. The rank $3$ version of this result was proved in \cite{FO12}. We record the following consequence of Theorem \ref{thm: rank four}:

\begin{cor}\label{merc4}
For $C\subset S$ with $g\geq 20$ and $\mathrm{Pic}(S)=\mathbb Z\cdot C$, we set $d:=\bigl\lfloor \frac{4g+14}{3} \bigr\rfloor$ and $A\in W^3_d(C)$ with $h^0(C, A)=4$. Then $E|_C$ is a stable rank $4$ bundle with $\gamma(E|_C)<\lfloor \frac{g-1}{2}\rfloor$. It follows that the statement $(M_4)$ fails for $C$.
\end{cor}

The curves $C$ appearing in Corollary \ref{merc4} are Brill-Noether general, that is, they satisfy $\mbox{Cliff}(C)=\lfloor \frac{g-1}{2} \rfloor$, see \cite{L1}.

Theorem \ref{thm: rank four} and Corollary \ref{merc4} fit into a more general set of results that are independent from the structure of $\mathrm{Pic}(S)$.
For example, we show that under mild restrictions,  on a very general $K3$ surface, the extension (\ref{eqn: E|_C}) is non-trivial and the restricted Lazarsfeld-Mukai bundle $E|_C$ is simple (see Theorem \ref{simple}). We expect that the  bundle $E|_C$  remains stable also for higher rank $r+1=h^0(C, A)$, at least when $\mbox{Pic}(S)=\mathbb Z\cdot C$. However, our method of proof based on the Bogomolov inequality, seems  not to extend easily for $r\geq 4$.

\vskip 5pt

The second topic we discuss in this paper concerns a connection between normal bundles of canonical curves and Mercat's conjecture. The question we pose is however fundamental and interesting irrespective of Mercat's conjecture. 

For a smooth non-hyperelliptic canonically embedded curve $C\subset \PP^{g-1}$ of genus $g$, we consider the normal bundle $N_{C}:=N_{C/\PP^{g-1}}$;  we then define the twist of the conormal bundle $E:=N_{C}^{\vee}\otimes K_C^{\otimes 2}$. By direct calculation
$$\mbox{det}(E)=K_C^{\otimes (g-5)} \  \mbox{ and } \ \mbox{rk}(E)=g-2.$$
In particular, the vector bundle $E$ contributes to $\mbox{Cliff}_{g-2}(C)$ if and only if $g\leq 8$.
Since $M_{K_C}(-1)=\Omega_{\PP^{g-1} | C}$, the bundle $E$ sits in the following exact sequence
\begin{equation}\label{eqr}
0\longrightarrow E\longrightarrow M_{K_C}\otimes K_C\stackrel{\gamma_{K_C}}\longrightarrow K_C^{\otimes 3}\longrightarrow 0,
\end{equation}
where $\gamma_{K_C}:H^0(C, M_{K_C}\otimes K_C)\rightarrow H^0(C, K_C^{\otimes 3})$ is the Gaussian map of $C$, see \cite{W}. The map $\gamma_{K_C}$ vanishes on symmetric tensors, hence
$\mbox{Ker}(\gamma_{K_C})=I_2(K_C)\oplus \mbox{Ker}(\psi_{K_C})$, where
$$\psi_{K_C}:=\gamma_{K_C |_{\wedge ^2 H^0(C, K_C)}}: \bigwedge^2 H^0(C, K_C)\rightarrow H^0(C, K_C^{\otimes 3}),$$
and $I_2(K_C)=K_{1,1}(C,K_C)$ is the space of quadrics containing the canonical curve $C$. The map $\psi_{K_C}$ has been studied intensely in the context of deformations
in $\PP^g$ of the cone over the canonical curve $C\subset \PP^{g-1}$, see \cite{W}. It is in particular known \cite{CHM}, \cite{V} that $\psi_{K_C}$ is surjective for a general curve $C$ of genus $g\geq 12$.

\vskip 3pt

We now specialize to the case $g=7$, when $E$ contributes to $\mbox{Cliff}_5(C)$. Then $$\mbox{rk}(E)=5\ \mbox { and } \mbox{det}(E)=K_C^{\otimes 2},$$ therefore $\mu(E)=\frac{24}{5}$. It is easy to show that the Gaussian map $\psi_{K_C}$ is injective for every smooth curve $C$ of genus $7$ having maximal Clifford index $\mbox{Cliff}(C)=3$. In particular, the space
$$H^0(C, E)=I_2(K_C)$$
is $10$-dimensional and $\gamma(E)=2+\frac{4}{5}<\mbox{Cliff}(C)$. We establish the following result:

\begin{thm}\label{g7}
The normal bundle $N_{C/\PP^6}$ of every canonical curve $C$ of genus $7$ with maximal Clifford index is stable. In particular, the Mercat conjecture $(M_5)$ fails for a general curve of genus $7$.
\end{thm}

The proof of Theorem \ref{g7} uses in an essential way Mukai's realisation \cite{M3} of a canonical curve $C$ of genus $7$ with $\mbox{Cliff}(C)=3$ as a linear section of the $10$-dimensional spinorial variety $OG(5,10)\subset \PP^{15}$. In particular,  the vector bundle $E$ is the restriction to $C$ of the rank $5$ spinorial bundle on $OG(5,10)$, which endows $E$ with an extra structure that only exists in genus $7$.  Note that the normal bundle of every canonical curve of genus at most $6$ is unstable, and more generally, the normal bundle of a tetragonal canonical curve of any genus is unstable (see also Section 3). In particular, we have the following identification of cycles on $\cM_7$
\begin{equation}\label{ident}
\Bigl\{[C]\in \cM_7: N_C \mbox{ is unstable}\Bigr\}=\cM_{7,4}^1,
\end{equation}
where the right hand side denotes the divisor of tetragonal curves of genus $7$.
We make the following conjecture:

\begin{con}
The normal bundle $N_{C/\PP^{g-1}}$ of a general canonical curve $C$ of genus $g\geq 7$ is stable.
\end{con}

Note that the stability of the normal bundle $N_{C/\PP^r}$ of a curve of genus $g$ is not known even in the case of a non-special embedding $C\hookrightarrow \PP^r$ given by a line bundle $L\in \mbox{Pic}(C)$ of large degree. This is in stark contrast with the case the kernel bundle $M_L=\Omega_{\PP^r| C}(1)$, whose stability easily follows by a filtration argument due to Lazarsfeld \cite{L2}. For some very partial results in this direction, see \cite{EL}. In general, one can show by degenerating a canonical curve $C\subset \PP^{g-1}$ to the transversal union of two rational normal curves in $\PP^{g-1}$ meeting transversally in $g+1$ points, that $N_{C/\PP^{g-1}}$ is not too unstable. Due to the fact that the slope
$\mu(N_{C/\PP^{g-1}})$ is not an integer, this simple minded technique does not seem to lead to a full solution, because one cannot expect to find a specialization in which the corresponding limit of the normal bundle is a direct sum of line bundles of the same degree. It is of course, natural to ask whether a generalization of the equality (\ref{ident}) exists for higher genus.

\vskip 6pt

\noindent{{\em Acknowledgments.} The second author is grateful to Shigeyuki Kondo for organizing a wonderful conference "Developments in moduli theory" in Kyoto in June 2013 to celebrate Mukai's 60th birthday. The first author was partly supported by
a grant of the Romanian Ministry of National Education, CNCS-UEFISCDI project number PN-II-ID-PCE-2012-4-0156,
and by a Humboldt fellowship. He thanks Humboldt Universit\"at zu Berlin and Max Planck Institut f\"ur Mathematik Bonn for hospitality during the preparation of this work. The second and third authors are partly financed by the Sonderforschungsbereich 647 "Raum-Zeit-Materie".}

\section{Simplicity of restricted Lazarsfeld-Mukai bundles}
\label{sec:simple}

We fix a  $K3$ surface $S$, a smooth  curve $C\subset S$ of genus $g$ and a globally generated linear series $A\in W^r_d(C)$, with $h^0(C, A)=r+1$.
Using the evaluation sequence (\ref{eqn: F}), we form the vector bundle $F=F_{C, A}$; by dualizing, we obtain an exact sequence for the dual bundle $E=E_{C, A}:=F_{C, A}^{\vee}$:
\begin{equation}
\label{eqn: E}
0\longrightarrow H^0(C, A)^{\vee}\otimes \OO_S\longrightarrow E_{C, A}\longrightarrow K_C\otimes A^{\vee}\longrightarrow  0.
\end{equation}
It is well-known \cite{M89}, \cite{L1} that $c_1(E)=[C]$ and $c_2(E)=d$; moreover
$h^0(S, F)=0$ and $h^1(S, E)=h^1(S, F)=0$. Finally, one also has that
$$\chi(S, E\otimes F)=2-2\rho(g, r, d);$$ in particular, if $E$ is a simple bundle, then $\rho(g, r, d)\geq 0$.
Assuming furthermore that $\mbox{Pic}(S)=\mathbb Z\cdot C$, it is also well-known that both $E$ and $F$ are $C$-stable bundles on $S$.

\subsection{The rank $2$ case}
We begin by showing that in rank $2$, irrespective of the structure of $\mbox{Pic}(S)$, a splitting of the restriction $E|_C$ can only be induced by an elliptic pencil on the $K3$ surface.

\begin{thm}
\label{thm: simple LM}
Let $C\subset S$ be as above and $A\in W^1_d(C)$ a base point free pencil of degree $2<d<g-1$ with $K_C\otimes A^{\vee}$
globally generated. The following conditions are equivalent:
\begin{enumerate}
\item[(i)] $E|_C\cong A\oplus (K_C\otimes A^{\vee})$;
\item[(ii)] There exists an elliptic pencil $N\in\mathrm{Pic}(S)$ such that
$N|_C= A$.
\item[(iii)] $h^0(S,E\otimes F)<h^0(C,E\otimes F|_C)$.
\end{enumerate}
\end{thm}

\begin{cor}
With notation as above, if $g\leq 2d-2$ and $A$ is not induced by an elliptic pencil on $S$,  then
$E|_C$ is simple if and only if $E$ is simple.
\end{cor}

Note that it is easy to see that if $E|_C$ is simple, then $E$ is also simple. It is also known that if $E$ is simple, then automatically $g\leq 2d-2$.

\proof (of Theorem \ref{thm: simple LM})
\noindent {\em (ii)$\Rightarrow$(i).}
Let $N$ be an elliptic pencil with $N|_C=A$. Consider the exact sequence
\[
0\longrightarrow N^{\vee}\longrightarrow F\longrightarrow N(-C)\longrightarrow 0.
\]
Its restriction to $C$ gives a splitting of the dual of the sequence (\ref{eqn: E|_C})
characterizing $E|_C$. Observe that since $d<g-1$, there is no morphism from $A^{\vee}$ to $K_C^{\vee}\otimes A$.

\medskip

\noindent {\em (i)$\Rightarrow$(ii).}
Conversely, suppose that $E|_C= A\oplus (K_C\otimes A^{\vee})$.
Applying $\mathrm{Hom}(K_C\otimes A^{\vee},\ -\ )$ to the sequence
(\ref{eqn: F}), we obtain an exact sequence
\[
0\longrightarrow \mathrm{Ext}^1(K_C\otimes A^{\vee}, F)\longrightarrow
\mathrm{Ext}^1(K_C\otimes A^{\vee} ,H^0(C, A)\otimes \OO_S)\longrightarrow
\mathrm{Ext}^1(K_C\otimes A^{\vee}, A).
\]
Since the extension class $[E]\in \mathrm{Ext}^1(K_C\otimes A^{\vee}, H^0(C, A)\otimes \OO_S)$
maps to the trivial extension in $\mathrm{Ext}^1(K_C\otimes A^{\vee}, A)$, it
follows  that there exists a rank $2$ bundle $G$ on $S$ which
fits into a commutative diagram:
\begin{equation}
\label{eqn: G}
\xymatrix{ &0\ar[d]&0\ar[d]&
\\
0\ar[r]& F\ar[r]\ar[d] &  H^0(A)\otimes \OO_S\ar[d]\ar[r]&A\ar@{=}[d]\ar[r]&0
\\
0\ar[r] & G \ar[d]\ar[r] &E \ar[d]\ar[r] & A\ar[r] &0
\\
 & K_C\otimes A^{\vee} \ar@{=}[r]\ar[d] & K_C\otimes A^{\vee} \ar[d] & &
\\
 & 0 &0 &}
\end{equation}

Using that $H^0(S, F)=H^1(S, F)=0$, we obtain
$H^0(S, G)\cong H^0(C, K_C\otimes A^{\vee})$. Since $h^0(S, E)=h^0(C, A)+h^1(C, A)=h^0(C, A)+h^0(S, G)$,
and $h^1(S, E)=0$, it follows that $H^1(S, G)=0$. From the second row of  (\ref{eqn: G}), we find that
$H^0(S, G(-C))=0$.

Furthermore, we compute
$c_1(G)=0$ and $c_2(G)=2d-2g+2$. So
$c_2(G)<0=c_1^2(G)$, that is, $G$ violates Bogomolov's
inequality, and then it sits in an extension
\begin{equation}
\label{eqn: G ext}
0\longrightarrow M\longrightarrow G\longrightarrow M^{\vee}\otimes \mathcal I_{\Gamma/S}\longrightarrow 0,
\end{equation}
where $\Gamma$ is a zero-dimensional subscheme of $S$,
and $M\in \mbox{Pic}(S)$ is such that $M^2>0$ and $M\cdot H>0$ for any ample
line bundle $H$ on $S$. In particular, $H^0(S, M^{\vee})=0$,
and hence $H^0(S, M)\cong H^0(S, G)\cong H^0(C, K_C\otimes A^{\vee})\ne 0$. Moreover, since $$h^0(S,M^\vee\otimes \mathcal I_{\Gamma/S})=h^1(S,G)=0,$$
 it also follows that $H^1(S,M)=0$.

On the other hand $H^0(S, F)=0$, which implies that the composed map
$$M\longrightarrow G\longrightarrow K_C\otimes A^{\vee}$$ is non-zero; in fact, we claim that it is surjective,
that is, $M|_C= K_C\otimes A^{\vee}$.
Suppose that $M|_C=K_C\otimes A^{\vee}(-D^\prime)$, with $D^\prime\ne 0$
an effective divisor on $C$.
Since $h^0(S, G(-C))=0$, we have $h^0(S, M(-C))=0$, which
implies $h^0(S, M)\le h^0(C, M|_C)$.
Since we assumed $K_C\otimes A^{\vee}$ to be globally generated, we have that
$$h^0(S, M)\le h^0(C, K_C\otimes A^{\vee}(-D^\prime))<h^0(C, K_C\otimes A^{\vee})=h^0(S, M),$$
 a contradiction.

Setting $N:=M^{\vee}(C)$, we have shown that $N|_C=A$ and there is an exact sequence
\[
0\longrightarrow M^{\vee}\longrightarrow N\longrightarrow A\longrightarrow 0.
\]
Since $h^0(S, M^{\vee})=h^1(S, M^{\vee})=0$, it follows that $h^0(S, N)=h^0(C, A)=2$ and hence
$N$ defines an elliptic pencil.

\medskip

\noindent {\em (iii)$\Rightarrow$(i).}
From the sequence (\ref{eqn: F}) twisted by $E(-C)\cong F$, we obtain that
\[
H^0(S, E\otimes F(-C))\subset H^0(C, A)\otimes H^0(S, E(-C)),
\]
and, since $F$ has no sections, it follows that $H^0(S, E\otimes F(-C))=0$.
We have an exact sequence
\[
0\longrightarrow H^0(S, E\otimes F)\longrightarrow  H^0(S, E\otimes F|_C)\longrightarrow H^1(S, E\otimes F(-C)).
\]
The hypothesis implies that $H^1(S, E\otimes F(-C))\ne 0$. From (\ref{eqn: F}) twisted by $E(-C)\cong F$, we obtain
the exact sequence in cohomology
\[
0\longrightarrow H^0(C, E|_C\otimes K_C^{\vee}\otimes A)\longrightarrow H^1(S, E\otimes F(-C))\longrightarrow
H^0(C, A)\otimes H^1(S, E(-C))=0,
\]
therefore $h^0(C, E|_C\otimes K_C^{\vee}\otimes A)\ne 0$. The sequence (\ref{eqn: E|_C})
yields to an exact sequence
\[
0=H^0(C, K_C^{\vee}\otimes A^{\otimes 2})\longrightarrow H^0(C, E|_C\otimes K_C^{\vee}\otimes A)\longrightarrow
H^0(C, \OO_C)\to H^1(C, K_C^{\vee}\otimes A^{\otimes 2}).
\]
Then $H^0(C, E|_C\otimes K_C^{\vee}\otimes A)\to
H^0(C, \OO_C)$ is an isomorphism and, under the coboundary map $$H^0(C, \OO_C)\ni 1\mapsto 0
\in H^1(C, K_C^{\vee}\otimes A^{\otimes 2}),$$ that is, the sequence (\ref{eqn: E|_C}) is split.

\medskip

Note that we also have $h^1(S, E\otimes F(-C))=1$ and
$h^0(C, E\otimes F|_C)=h^0(S, E\otimes F)+1$.

\medskip

\noindent {\em (i)$\Rightarrow$(iii).}
From the hypothesis and from the sequence
(\ref{eqn: E|_C}), we find
$$h^0(C, E|_C\otimes A^{\vee})=h^0(C, K_C\otimes A^{\otimes (-2)})+1.$$
Furthermore, $h^0(S, E\otimes F)=h^0(C, E|_C\otimes A^{\vee})$;
twist (\ref{eqn: E}) by $F$ and use the vanishing
of $h^0(F)$ and that of $h^1(F)$.

On the other hand, since $E|_C\cong A\oplus K_C\otimes A^{\vee}$,
we have
$$h^0(C, E\otimes F|_C)=2+h^0(C, K_C\otimes A^{\otimes (-2)}),$$
hence $h^0(C, E\otimes F|_C)=h^0(S, E\otimes F)+1$.
\endproof

\vskip 4pt

\subsection{Lazarsfeld-Mukai bundles of higher rank}

We study when the restriction $E|_C$ is a simple vector bundle. Our main tool is a variant of the Bogomolov instability theorem.

\begin{thm}\label{simple}
Let $S$ be a $K3$ surface and $C\subset S$ a smooth curve of genus $g\geq 4$ such that  $\mathrm{Pic}(S)=\mathbb{Z}\cdot C$.
We fix positive integers $r$ and $d$ such that
$$\rho(g,r,d)\ge 0, \ g\geq 2r+4 \mbox{ and } d\leq \frac{3r(g-1)}{2r+2}.$$
Then for any linear series $A\in W^r_d(C)$ such that $h^0(C, A)=r+1$ and $K_C\otimes A^{\vee}$ is globally generated,
the restricted Lazarsfeld-Mukai bundle $E|_C$ is simple.
\end{thm}

Note that in the special case $\rho(g, r, d)=0$, the constraints
from the previous statement give rise to the bound $g>2r+5$.

\proof
{\em Step 1.} We first establish that the natural extension (\ref{eqn: E|_C}), that is,
$$
0\longrightarrow Q_A \longrightarrow E|_C\longrightarrow K_C\otimes A^{\vee}\longrightarrow 0
$$
is non-trivial. Assuming that (\ref{eqn: E|_C}) is trivial. Then there is an
injective morphism from $K_C\otimes A^{\vee}$ to $E|_C$ and hence
a surjective map
$F(C)\to A$. Then
$$G:=\mbox{Ker}\{F(C)\to A\}$$ is a vector
bundle of rank $r+1$ with Chern classes
$c_1(G)=(r-1)[C]$ and
$$
c_2(G)=c_2(F(C))-c_1(F(C))\cdot C+\deg(A)=2d+r(r-3)(g-1).
$$
We compute the discriminant of $G$
$$
\Delta(G)=2\rk(G)c_2(G)-(\rk(G)-1)c_1^2(G)=4d(r+1)-8r(g-1)<0,
$$
hence $G$ is unstable. Applying \cite{HL02} Theorem 7.3.4, there exists
a subsheaf $M\subset G$ with
$$
\xi_{M,G}^2\ge-\frac{\Delta(G)}{r(r+1)^2},
$$
where
$
\xi_{M,G}=c_1(M)/\rk(M)-c_1(G)/\rk(G).
$ Setting $c_1(M)=k\cdot [C]$ and $s:=\rk(M)$, the previous inequality becomes
$$
\left(\frac{k}{s}-\frac{r-1}{r+1}\right)^2(2g-2)\ge \frac{8r(g-1)-4d(r+1)}{r(r+1)^2}.
$$

Note that $M$ destabilizes $G$, which coupled with the stability of $F(C)$ yields
$$
\frac{r-1}{r+1}\leq \frac{k}{s}< \frac{r}{r+1},
$$
implying after manipulations  $2d(r+1)>3(g-1)r$, thus contradicting the hypothesis.

\medskip

{\em Step 2.} Assuming that $E|_C$ is non-simple, we
deduce that the extension (\ref{eqn: E|_C}) splits. We consider the exact sequence
$$
H^0(S, E\otimes F)\longrightarrow H^0(C, E\otimes F|_C)\longrightarrow H^1(S, E\otimes F(-C)).
$$
and it suffices to show that $H^1(S, E\otimes F(-C))=0$. Assuming this not to be the case, twisting (\ref{eqn: F}) by $E(-C)$ induces
the exact sequence
$$
H^0(C, A\otimes E|_C\otimes K_C^{\vee})\longrightarrow H^1(S, E\otimes F(-C))\longrightarrow H^0(C, A)\otimes H^1(S, E(-C)).
$$

Since $H^1(S, E(-C))=0$, we obtain that $H^0(C, A\otimes E|_C\otimes K_C^{\vee})\ne 0$.
Furthermore, $Q_A$ is a stable bundle and since $\mu(Q_A\otimes A\otimes K_C^{\vee})<0$, we find that
$$H^0(C, Q_A\otimes A\otimes K_C^{\vee})=0,$$ hence we also have the sequence
induced from (\ref{eqn: E|_C}) after twisting with $A\otimes K_C^{\vee}$
$$
0\longrightarrow  H^0(C, E|_C\otimes K_C^{\vee}\otimes A)\longrightarrow
H^0(C, \mathcal O_C)\longrightarrow H^1(C, K_C^{\vee}\otimes A\otimes Q_A).
$$
We conclude that the coboundary map $H^0(C, \mathcal O_C)\to H^1(C, Q_A\otimes A\otimes K_C^{\vee})$
is trivial, that is, $E|_C\cong Q_A\oplus (K_C\otimes A^{\vee})
$, which completes the proof.
\endproof

\section{Stability of restricted Lazarsfeld-Mukai bundles}
\label{sec:stable}

\subsection{The rank $2$ case} If $C\subset S$ is an ample curve, then with one exception ($g=10$ and $C$ a smooth plane sextic),  $\mbox{Cliff}(C)$ is computed by a pencil, see \cite{CP95} Proposition 3.3. We show that in rank $2$ the semistability of the LM bundle is preserved under restriction.

\begin{thm}
\label{thm: stable LM}
Let $S$ be a $K3$ surface, $C\subset S$ an ample curve of genus $g\geq 4$ and $A\in W^1_d(C)$ a pencil computing $\mathrm{Cliff}(C)$.
If $E_{C,A}$ is $C$-semistable on $S$, then $E|_C$ is also semistable on $C$. Moreover, if $E_{C,A}$ is $C$-stable on $S$, then $E|_C$ is stable on $C$.
\end{thm}

\proof
The proof of the stability is similar, and hence we discuss the semistability part only.
We write $A=\OO_C(D)$, where $D$ is an effective divisor on $C$.  Suppose $E|_C$ is unstable and consider an exact sequence
\[
0\longrightarrow L_1\longrightarrow E|_C\longrightarrow K_C\otimes L_1^{\vee}\longrightarrow 0,
\]
 with
$\mathrm{deg}(L_1)\ge g$. Since $L_1\nsubseteq A$, the composed map
$L_1\to E|_C\to K_C\otimes A^{\vee}$
must be non-zero, that is, $L_1=K_C(-D-D_1)$,
where $D_1$ is an effective divisor on $C$. Set $d_1:=\mathrm{deg}(D_1)$. Consider the elementary modification
\begin{equation}
\label{eqn: V}
0\longrightarrow V\longrightarrow E\longrightarrow A(D_1)\longrightarrow 0
\end{equation}
induced by the composition
$E\to E|_C\to A(D_1)$. Then
$$c_1(V)=0\  \mbox{ and    } \ c_2(V)=2d+d_1-2g+2<0,$$ hence $V$ is unstable with respect to any polarization and fits in an exact sequence
\begin{equation}
\label{eqn: V ext}
0\longrightarrow M\longrightarrow V\longrightarrow M^{\vee}\otimes \mathcal I_{\Gamma/S} \longrightarrow 0,
\end{equation}
where $\Gamma\subset S$ is a $0$-dimensional subscheme and $M$ is a divisor class that intersects positively any ample class on $S$ and with $M^2>0$.
From (\ref{eqn: V}) and (\ref{eqn: V ext}) we find that $H^0(S, M)\cong H^0(S, V)$ and $H^0(S, M(-C))=0$.
Dualizing (\ref{eqn: V}), we obtain the sequence
\[
0\longrightarrow F\longrightarrow V^{\vee}\longrightarrow K_C(-D-D_1)\longrightarrow 0,
\]
from which, using that $V\cong V^{\vee}$, we obtain $H^0(S, V)=H^0(C, K_C(-D-D_1))$.
\vskip 3pt

We claim that $\mathrm{Cliff}(A(D_1))=\mathrm{Cliff}(C)$.
Recall that $h^0(S, E)=h^0(C, A)+h^1(C, A)$, and, from the sequence
(\ref{eqn: V}) we write
$$h^0(S, E)\le h^0(C, A(D_1))+h^1(C, A(D_1)).$$
By assumption, the pencil $A$ computes $\mathrm{Cliff}(C)$, which implies
\[
\mathrm{Cliff}(C)=g+1-h^0(A)-h^1(A)\ge g+1-h^0(A(D_1))-h^1(A(D_1))
=\mathrm{Cliff}(A(D_1)).
\]
It follows that $\mathrm{Cliff}(A(D_1))=\mathrm{Cliff}(C)$, in particular $K_C(-D-D_1)$ is globally
generated.

\vskip 3pt

Clearly, $M\nsubseteq F$, hence the composition  $\varphi: M\to V\to K_C(-D-D_1)$ is non-zero and one writes $\mbox{Im}(\varphi)=K_C(-D-D_1-D_2)$, where $D_2$
is an effective divisor on $C$.
If $D_2\ne 0$, then one has the sequence of inequalities
\[
h^0(S, M)\le h^0(C, K_C(-D-D_1-D_2))<h^0(C, K_C(-D-D_1))=h^0(S, M),
\]
a contradiction. Therefore $M|_C=K_C(-D-D_1)$. Viewing $M$ as a subsheaf of $E$, we find   $\mu(M)=M\cdot C=\mathrm{deg}(L_1)>\mu(E)$, thus bringing the proof to an end.
\endproof

\begin{rmk}
If $E_{C, A}$ is stable, then it is simple  and hence $d=\lfloor \frac{g+3}{2}\rfloor $, see \cite{L1}. Conversely, if $C'\subset S$ is an ample curve of genus $g$ and gonality $\lfloor \frac{g+3}{2}\rfloor$, then it was shown in \cite{LC} that the LM bundle $E_{C,A}$ corresponding to a general curve $C\in |\OO_S(C')|$ and a pencil $A\in W^1_{\lfloor \frac{g+3}{2}\rfloor}(C)$ is $C$-semistable (even stable when $g$ is odd).
\end{rmk}

\subsection{Stability of Lazarsfeld-Mukai bundles of rank four}

We show that restrictions of LM bundles of rank $4$ on very general $K3$ surfaces of genus $g\geq 20$ are stable. Similar results were established in \cite{V} and \cite{FO12} for rank $2$ and $3$ respectively. We fix integers $i\geq 6$ and  $\rho\geq 0$ and write
$$g:=4i-4+\rho \ \ \mbox{ and } \ \ d:=3i+\rho,$$
so that $\rho(g, 3, d)=\rho$.  Let $S$ be a $K3$ surface and $C\subset S$ a curve of genus $g$ such that  $\mbox{Pic}(S)=\mathbb Z\cdot C$, and pick a globally generated linear series $A\in W^3_d(C)$ with $h^0(C, A)=4$.

\vskip 5pt

\noindent \emph{Proof of Theorem \ref{thm: rank four}.} Our previous results show that $E|_C$ is simple, hence indecomposable. Suppose $E|_C$ is not stable and fix a maximal destabilizing sequence
\[
0\longrightarrow M\longrightarrow E|_C\longrightarrow N\longrightarrow 0.
\]

Put $d_N:=\mathrm{deg}(N)$ and $d_M:=\mathrm{deg}(M)=2g-2-d_N$. Since $M$ is destabilizing,
\begin{equation}
 \label{eqn: M destab}
\frac{d_M}{\mathrm{rk}(M)}\ge \frac{g-1}{2},\ \ \ \frac{d_N}{\mathrm{rk}(N)}\le \frac{g-1}{2}.
\end{equation}

The bundle $N$, being a quotient of $E$, is globally generated. Since $H^0(C, E|_C^{\vee})=0$, clearly $N\neq \OO_C$, therefore $h^0(C, N)\geq 2$. From the inequalities (\ref{eqn: M destab}) it follows that $\mbox{rk}(N)>1$, because $C$ has maximal gonality.

\medskip

{\em Step 1.} We prove that $M$ is a line bundle. Assume that, on the contrary,
$$\mbox{rk}(M)=\mbox{rk}(N)=2$$ and
consider the elementary modification $G:=\mathrm{Ker}\{E\to N\}$.
Its Chern classes are given as follows:
$$
 c_1(G)=-[C],\ \  \
c_2(G)=d+d_N-2(g-1),
$$
and its discriminant equals
$\Delta(G)=-64 i + 110 + 8 d_N-14\rho<0$,
because of (\ref{eqn: M destab}). In particular, there exists a saturated subsheaf $F\subset G$
which verifies the inequalities
\begin{equation}
\label{eqn: destab}
\mu(G)\le \mu(F)<\mu(E), \ \ \mbox{ and }
\end{equation}
\begin{equation}
\label{eqn: HL}
 \xi^2_{F,G}\ge -\frac{\Delta(G)}{48}.
\end{equation}

Write $c_1(F)=\alpha \cdot [C]$ and $\mathrm{rk}(F)=\beta\le 3$. The above
inequality (\ref{eqn: HL}) becomes
\[
  \left(\frac{\alpha}{\beta}+\frac{1}{4}\right)^2(2g-2)
\ge -\frac{\Delta(G)}{48}.
\]

We apply (\ref{eqn: destab}) for $\mu(F)=\alpha(2g-2)/\beta$ and obtain
$$-\frac{1}{4}\leq \frac{\alpha}{\beta}<\frac{1}{4},$$
hence $\alpha =0$, and the inequality (\ref{eqn: HL}) reads in this case
$
d_N\ge 5i-10+\rho.
$ Recalling that $d_N\le g-1=4i-5+\rho$, we obtain a contradiction whenever $i\ge 6$.

\medskip

{\em Step 2.} We construct an elementary modification, in order to reach a contradiction.

\medskip

From (\ref{eqn: M destab}), we have $d_M\ge \frac{g-1}{2}$.
The composite map $M\rightarrow E|_C\rightarrow K_C\otimes A^{\vee}$ is not zero, for else $M\hookrightarrow Q_A$ and since $\mu(Q_A\otimes M^{\vee})<0$, one contradicts the semistability of $Q_A$. We set $A_1:=K_C\otimes A^{\vee}\otimes M^{\vee}$ and obtain a surjection $F(C){|_C}\to A\otimes A_1$
inducing, as before, an elementary modification
$$V:=\mbox{Ker}\{F(C)\rightarrow A\otimes A_1\}.$$

By direct computation we show that $\Delta(V)<0$.
Indeed, we compute
$$c_1(V)=2\cdot [C],\ \ \ c_2(V)=d+2g-2-d_M, \ \ \mbox{ hence }
$$
$$
 \Delta(V)=8c_2(V)-3c_1^2(V)=8(d-d_M-g+1)=
8(5-d_M-i)<0.
$$

We obtain a destabilizing sheaf $P\subset V$,
with $\rk(P)=b\le 3$ and $c_1(P):=a\cdot [C]$, such that the following inequalities are both satisfied
\begin{equation}
\label{eqn: HL 2}
   \left(\frac{a}{b}-\frac{1}{2}\right)^2(2g-2)\ge
-\frac{\Delta(V)}{48} \ \ \  \mbox{ and } \ \ \ \mu(V)\le\mu(P)<\mu(F(C)).
\end{equation}

The second inequality gives $\frac{1}{2}\le \frac{a}{b}< \frac{3}{4}$, which leaves two possibilities: either $a=1$ and $b=2$, when via (\ref{eqn: HL 2}) one finds
that $\Delta(V)\geq 0$, a contradiction, or else $a=2$ and $b=3$,
when inequalities (\ref{eqn: HL 2}) and (\ref{eqn: M destab}) clash.
\hfill
$\Box$

\section{Normal bundle of canonical curves of genus $7$}

The aim of this section is to prove Theorem \ref{g7} and we begin by recalling Mukai's results \cite{M3} on canonical curves of genus $7$. We choose a vector space $U:=\mathbb C^{10}$ and a non-degenerate quadratic form $q:U\rightarrow \mathbb C$, defining a  smooth $8$-dimensional quadric
$Q\subset \PP(U)=\PP^9$.

The algebraic group $\mbox{{\bf Spin}}(U)$ corresponding to the Dynkin diagram $D_5$ admits two $16$-dimensional half-spin representations $\ss^+$ and $\ss^-$, which correspond to maximal weights $\alpha^+$ and $\alpha^-$ respectively. The homogeneous spaces $V^{\pm}:=\mbox{{\bf Spin}}(U)/P(\alpha^{\pm})$ are both $10$-dimensional and can be realized as the two irreducible components of the Grassmannian $G_q(5, U)$ of projective $4$-planes inside $\PP(U)$ which are isotropic with respect to the quadratic form $q$. From now on, we set
$$V:=V^+\subset \PP(\ss^+)=\PP^{15}.$$
Note that $\mbox{Aut}(V)=SO(10)$. If $\E$ is the restriction to $V$ of the tautological bundle on $G(5, 10)$, one has an exact sequence of vector bundles on $V$:
\begin{equation}\label{spin}
0\longrightarrow \E^{\vee}\longrightarrow U\otimes \OO_V\longrightarrow \E\longrightarrow 0.
\end{equation}
By the adjunction formula, smooth curvilinear sections of $V$ are canonical curves of genus $7$ and Mukai \cite{M3} showed that \emph{each} curve $[C]\in \cM_7$ with $\mbox{Cliff}(C)=3$ appears in this way. Precisely, there is a birational map
$$\alpha:G(7, 16)\dblq SO(10)\dashrightarrow \mm_7, \ \ \alpha(\Lambda):=[\Lambda\cap V],$$
where $\Lambda\cong \PP^6$. Given a curve $[C]\in \cM_7$, the inverse $\alpha^{-1}([C])$ is constructed precisely via the twist of the conormal bundle on $C$ mentioned in the introduction.
\vskip 3pt

Let $C\subset \PP^{6}$ be a smooth canonical curve with $\mbox{Cliff}(C)=3$, and set $E:=N_{C\PP^6}^{\vee}(2)$. One has an identification $H^0(C, E)=I_2(K_C)$ and $E$ is a globally generated
bundle. The tautological map
$$\phi_E:C\rightarrow G(5, H^0(C,E))$$ is easily shown to be injective and its image lies on $V$. In particular, the vector bundle $E$ is the restricted spinorial bundle, that is, $E=\E_{|C}$ and one has an exact sequence:
\begin{equation}\label{spin2}
0\longrightarrow E^{\vee}\longrightarrow H^0(C, E)\otimes C\longrightarrow E\longrightarrow 0.
\end{equation}
Note that $W^1_4(C)=\emptyset$, while $W^1_5(C)$ is a curve. We are going to make essential use of the following fact:
\begin{lem}
Let $C$ as above and $A\in W^1_5(C)$. Then there are no surjections $E\twoheadrightarrow A$.
\end{lem}\label{surj}
\begin{proof} We proceed by contradiction. Assume that there is such a pencil $A\in W^1_5(C)$, then use the base point free pencil trick to write the following diagram:
\begin{equation}
\xymatrix{
0\ar[r]& E^{\vee}\ar[r]\ar[d] &  H^0(C,E)\otimes \OO_C\ar[d]\ar[r]&E\ar[d]\ar[r]&0
\\
0\ar[r] & A^{\vee} \ar[r] &H^0(C,A)\otimes \OO_C \ar[r] & A\ar[d]\ar[r] &0
\\
 &  &  & 0&
}
\end{equation}
In particular, $H^0(C, E\otimes A^{\vee})\neq 0$. Via the identification $H^0(C,E)=I_2(K_C)$, this implies that if $L:=K_C\otimes A^{\vee}\in W^2_7(C)$, then the multiplication map
$$\mbox{Sym}^2 H^0(C,L)\rightarrow H^0(C, L^{\otimes 2})$$ is not injective. This is possible only if $L$ is not birationally very ample, in particular, $C$ must be trigonal, which is not the case.
\end{proof}

\vskip 3pt

We are now in a position to prove that the twist $E$ of the conormal bundle of a canonical curve of genus $7$ is  stable.

\vskip 5pt

\noindent \emph{Proof of Theorem \ref{g7}.} Suppose that $0\rightarrow F\rightarrow E\rightarrow M\rightarrow 0$ is a destabilizing sequence for the vector bundle $E$, that is, with
$\mu(F)\geq \mu(E)=\frac{24}{5}$. Since $E$ is globally generated, so is any of its quotient, in particular $M$ too. We distinguish several possibilities, depending on the ranks that appear:

\vskip 4pt

\noindent {\bf (i)} $\mbox{rk}(F)=4$ and $M$ is line bundle. Then $\mbox{deg}(F)\geq 20$, hence $\mbox{deg}(M)\leq 4$. Since $C$ is not tetragonal, $h^0(C,M)\leq 1$.  Note that $M\neq \OO_C$, for $H^0(C, E^{\vee})=0$. It follows that $M$ is not globally generated, a contradiction.

\vskip 3pt

\noindent {\bf (ii)} $\mbox{rk}(F)=1$ and we may assume that $\mbox{deg}(F)=5$. Suppose first that $h^0(C, F)=0$, therefore $h^0(C, K_C\otimes F^{\vee})=1$, and hence $K_C\otimes F^{\vee}$ is not globally generated. Since one has a surjection $E^{\vee}(1)\twoheadrightarrow K_C\otimes F^{\vee}$, we reach a contradiction by observing that $E^{\vee}(1)$ is globally generated. Indeed, via Serre duality, this last statement is equivalent to the equality $h^0(C, E(p))=h^0(C,E)=10$, for every point $p\in C$. From the exact sequence
$$0\longrightarrow E(p)\longrightarrow M_{K_C}\otimes K_C(p)\longrightarrow K_C^{\otimes 3}(p)\longrightarrow 0,$$
we obtain that $H^0(C, E(p))=\mbox{Ker}\Bigl\{H^0(C, M_{K_C}\otimes K_C(p))\rightarrow H^0(C, K_C^{\otimes 3}(p))\Bigr\}$.
The conclusion follows, since $H^0(C, M_{K_C}\otimes K_C)=H^0(C, M_{K_C}\otimes K_C(p))$.

\vskip 5pt

Suppose now that $h^0(C, F)\geq 1$. The case $h^0(C, F)\geq 2$ having been discarded in the course of proving Lemma \ref{surj}, we assume that $h^0(C, F)=1$, hence $h^0(C,K_C\otimes F^{\vee})=2$.
We obtain that the multiplication map 
$$\mbox{Sym}^2 H^0(C, K_C\otimes F^{\vee})\rightarrow H^0(C, K_C^{\otimes 2}\otimes F^{\otimes (-2)})$$ is not injective, which contradicts the base point free pencil trick.

\vskip 5pt

\noindent {\bf (iii)} $\mbox{rk}(F)=3$, and then $\mbox{deg}(F)\geq 15$, hence $\mbox{deg}(M)\leq 9$. This time we may assume that $F$ is stable. If $M$ is not stable, we choose a line subbundle $A\subset M$ of maximal degree, which we pull-back under the surjection $E\twoheadrightarrow M$, to obtain the exact sequence
$$0\longrightarrow G\longrightarrow E\longrightarrow M/A\longrightarrow 0.$$
We obtain that $\mbox{deg}(M/A)\leq \mbox{deg}(M)/2\leq 9/2$, that is, $\mbox{deg}(M/A)\leq 4$. In particular, $M/A$ is not globally generated, which is again a contradiction, so we can assume that both $F$ and $M$ are stable vector bundles. Since $h^0(C,M)+h^0(C, F)\geq h^0(C,E)=10$, the strategy is to use the fact that the  Mercat statements $(M_2)$ and $(M_3)$ have been established for curves $C$ of genus $7$ with maximal Clifford index, that is,
$$\mbox{Cliff}_2(C)=\mbox{Cliff}_3(C)=3,$$
see \cite{LN3} Theorem 4.5. In particular, if both $F$ and $M$ contribute to their respective Clifford indices, that is, $h^0(C,F)\geq 6$ and $h^0(C,M)\geq 4$ respectively, then we write
$$\frac{9}{2}+3\leq \frac{3}{2} \gamma(F)+\gamma(M)=\frac{1}{2}\Bigl(\mbox{deg}(F)+\mbox{deg}(M)\Bigr)-h^0(C, F)-h^0(C, M)+5,$$
that is, $h^0(C, F)+h^0(C,M)\leq \frac{19}{2}$, a contradiction.

\vskip 3pt

Assume now that one of the bundles $F$ or $M$ does not contribute to its Clifford index. Since $M$ is globally generated, $h^0(C,M)\geq 2$. We can have $h^0(C,M)=2$, only when $M=\OO_C^{\oplus 2}$, which is impossible, for $\OO_C^{\oplus 2}$ is not a direct summand of $E$. If $h^0(C,M)=3$, then $\mbox{deg}(M)\geq 7$, and one has equality if and only if $M=Q_L$, where $L\in W^2_7(C)$. Assuming this to be the case, we choose two points $p, q\in C$ that correspond to a node in the plane model $\phi_L:C\rightarrow \PP^2$, that is, $A:=L(-p-q)\in W^1_5(C)$. Then there is a surjection $Q_L\twoheadrightarrow A$, which by composition gives rise to a surjective morphism $E\twoheadrightarrow A$. This contradicts Lemma \ref{surj}. 

Thus we may assume that $\mbox{deg}(M)\geq 8$, and accordingly, $\mbox{deg}(F)\leq 16$. Then we compute
$$\gamma(F)=\mu(F)-\frac{2}{3} h^0(C,F)+2\leq \frac{16}{3}-\frac{14}{3}+2<\mbox{Cliff}(C),$$
which again contradicts the equality $\mbox{Cliff}_3(C)=3$.

\vskip 4pt

\noindent {\bf (iv)} $\mbox{rk}(F)=2$, and then $\mbox{deg}(F)\geq 10$ and $\mbox{deg}(M)\leq 14$. We may assume this time that $M$ is stable. If $F$ is not stable, then it has a line subbundle $A\hookrightarrow F$ with $\mbox{deg}(A)\geq 5$, and we are back to case (ii). Thus both $M$ and $F$ are stable bundles, and we proceed precisely like in case (iii).

\hfill $\Box$

\vskip 5pt

It is instructive to remark that the normal bundle of a canonical curve of genus $g<7$ is never stable. More generally we have the following:
\begin{prop}
The normal bundle of a tetragonal canonical curve of genus $g$ is unstable.
\end{prop}
\begin{proof} More generally, we begin with a $k:1$ covering $f:C\rightarrow \PP^1$, and consider the rank $(k-1)$-vector bundle $\F^{\vee}:=f_*\OO_C/\OO_{\PP^1}$ on the projective
line. Then $\pi:X=\PP(\F)\rightarrow \PP^1$ is a scroll of dimension $k-1$, which contains the canonical curve $C$ and which can be embedded by the tautological bundle
$\OO_{X}(1)$ in $\PP^{g-1}$ as a variety of degree $g-k+1$. Denoting by $H, R\in \mbox{Pic}(X)$ the class of the hyperplane section and that of the ruling respectively,
we have $$K_X\equiv -(k-1)H+(g-k-1)R,$$
whereas obviously $C\cdot H=2g-2$ and $C\cdot R=k$. We compute the degree of the normal bundle $N_{C/X}$ and find:
$$\mbox{deg}(N_{C/X})=\mbox{deg}(T_{X|C})+\mbox{deg}(K_C)=k(g+k-1).$$
We write the usual exact sequence relating normal bundles
$$0\longrightarrow N_{C/X}\longrightarrow N_{C/\PP^{g-1}}\longrightarrow N_{X/\PP^{g-1}}\otimes \OO_C \longrightarrow 0,$$
and compare the slopes
$$\mu(N_{C/X})=\frac{k(g+k-1)}{k-2}  \ \ \mbox{ and } \ \ \mu(N_{C/\PP^{g-1}})=\frac{2(g-1)(g+1)}{g-2}.$$
We conclude that for $k=4$ and $g\geq 6$, the normal bundle $N_{C/X}$ is a destabilizing subbundle of  $N_{C/\PP^{g-1}}$. For $g$ at most $5$, every canonical curve of genus $g$ is a complete intersection which obviously produces a destabilizing line subbundle.
\end{proof}

\end{document}